\documentclass[a4,intlimits,openany,12pt]{article} 
\usepackage[cp1251]{inputenc}
\usepackage[T2B]{fontenc}
\usepackage[russian]{babel}
\usepackage{amssymb}
\usepackage{amsmath}
\usepackage{latexsym}
\usepackage{graphicx}
\usepackage{mathrsfs}
\usepackage{amscd}
\usepackage{algorithm}
\usepackage{algorithmic}
\usepackage{multirow}

\usepackage{amsthm,  amsfonts}
\numberwithin{equation}{section}

\newcounter{myenumcount}
{\end{list}}
\makeindex

\theoremstyle{plain}
\newtheorem{theorem}{Теорема}
\newtheorem{corollary}{Следствие}

\newtheorem{lemma}{Лемма}
\theoremstyle{definition}

\newtheorem{remark}{Замечание}

\usepackage{ccaption}
\captiondelim{. }

\begin{document}
	\begin{center}
		{\large О НЕКОТОРЫХ АДАПТИВНЫХ АЛГОРИТМАХ ЗЕРКАЛЬНОГО СПУСКА ДЛЯ ЗАДАЧ ВЫПУКЛОЙ И СИЛЬНО ВЫПУКЛОЙ ОПТИМИЗАЦИИ С ФУНКЦИОНАЛЬНЫМИ ОГРАНИЧЕНИЯМИ}
	\end{center}

\begin{center}
Ф. С. Стонякин, М. С. Алкуса, А. А. Титов 		
\end{center}

\section{Введение}
Задачи минимизации выпуклых гладких и негладких функционалов c ограничениями возникают во многих задачах современной large-scale оптимизации и её приложений \cite{bib_ttd,bib_Shpirko_Nesterov}. Для таких задач имеется множество методов, среди которых можно выделить метод уровней \cite{bib_Nesterov}, метод штрафных функций \cite{bib_Lan,bib_Vasilyev}, метод множителей Лагранжа \cite{bib_Boyd}. Метод зеркального спуска (МЗС) \cite{bib_mirror_and_projected,bib_Nemirovski1983} восходит к обычному градиентному спуску и вполне может считаться достаточно простым методом для задач негладкой выпуклой оптимизации. Предлагаемая работа посвящена некоторым адаптивным методам зеркального спуска для задач выпуклого программирования с липшицевыми функциональными ограничениями.

Отметим, что функциональные ограничения, вообще, могут быть негладким (недифференцируемыми) и поэтому мы рассматриваем субградиентные методы. Методы с использованием субградиентов негладких выпуклых функций разрабатываются уже несколько десятилетий и восходят к известным пионерским работам, одна из которых посвящена градиентному методу для безусловных задач при евклидовом расстоянии \cite{bib_Shor}, а другая --- его обобщению для задач с ограничениями \cite{bib_Polyak1967}. В работе \cite{bib_Polyak1967} предложена идея переключения шагов между направлением субградиента целевого функционала и направлением субградиента ограничения. Обобщение метода градиентного спуска на постановку задачи с неевклидовым расстоянием называют {\it методом зеркального спуска}. Этот метод был предложен в \cite{bib_Nemirovski1979,bib_Nemirovski1983} (см. также \cite{bib_mirror_and_projected}). Зеркальный спуск для задач с функциональными ограничениями был предложен в \cite{bib_Nemirovski1983} (см. также \cite{bib_comirror}). При этом, как правило, для нахождения величины шага и критерия остановки для зеркального спуска необходимо знать величину константы Липшица целевого функционала, а также ограничения. Известны также и методы с адаптивным выбором шага, рассмотрены в \cite{bib_Nemirovski} для задач без ограничений, а в \cite{bib_comirror} --- для задач с ограничениями. Недавно в \cite{bib_Adaptive} были предложены оптимальные алгоритмы зеркального спуска для задач выпуклого программирования с липшицевыми функциональными ограничениями с адаптивным выбором шага, а также адаптивными критериями остановки. Также модификации этих методов для задач в случае нескольких выпуклых функциональных ограничений были проанализированы в \cite{bib_Stonyakin}.

В настоящей статье мы рассматриваем некоторые алгоритмы зеркального спуска для задач минимизации выпуклого функционала $f$ с неположительным, выпуклым и липшицевым негладким функциональным ограничением $g$. Важно, что целевой функционал может иметь разный уровень гладкости. В частности, целевой функционал $f$ может не удовлетворять свойству Липшица, но иметь липшицев градиент. Например, квадратичные функционалы не удовлетворяют обычному свойству Липшица (или константа Липшица достаточно большая), но имеют липшицев градиент. Можно рассматривать и негладкие выпуклые функции, равные максимуму конечного набора дифференцируемых функционалов с липшицевым градиентом. Например, пусть $A_{i} (i \in \overline{1, m})$ --- положительно полуопределённые матрицы ($x^{T}A_{i}x \geq 0 $ для всякого $x\in X $) и целевой функционал имеет вид
\begin{equation}\label{equiv_nonstand1}
f(x)=\max\limits_{i = \overline{1, m}}f_{i}(x),
\end{equation}
где
\begin{equation}\label{equiv_nonstand2}
f_{i}(x) = \dfrac{1}{2} \langle A_{i}x,x \rangle - \langle b_{i},x \rangle + c_{i}, \;\;  i = \overline{1, m}.
\end{equation}
для некоторых фиксированных $b_{i} \in \mathbb{R}^{n} $ и $c_{i} \in \mathbb{R}$, для всех $i = \overline{1, m}$.
Отметим, что функционалы вида \eqref{equiv_nonstand1} -- \eqref{equiv_nonstand2} возникают в задачах проектирования механических конструкций Truss Topology Design со взвешенными балками \cite{bib_Nesterov2016}. Для задач минимизации функционалов такого типа  при наличии выпуклых липшицевых ограничений в \cite{bib_Adaptive,bib_Titov,bib_Stonyakin} на базе методики работ Ю.Е. Нестерова \cite{bib_Nesterov,bib_Nesterov2016} были предложены некоторые новые адаптивные алгоритмы зеркального спуска, а также обоснована их оптимальность. Часть этих результатов (про частично адаптивный метод) была заявлена в качестве доклада на VII Международную конференцию <<Проблемы оптимизации и их приложения>> (OPTA-2018) \cite{bib_Titov}. Настоящая статья посвящена изложению основных результатов доклада \cite{bib_Titov}, а также развитию результатов \cite{bib_Adaptive,bib_Titov,bib_Stonyakin} в следующих направлениях.

Во-первых, доказывается оптимальность с точки зрения оракульных оценок предложенных методов в \cite{bib_Adaptive,bib_Titov,bib_Stonyakin} для задач с выпуклым липшицевым целевым функционалом, а также для задач с липшицевым гессианом при наличии выпуклых липшицевых ограничений.

Во-вторых, на базе техники рестартов (перезапусков) методов из \cite{bib_Adaptive,bib_Titov} (для выпуклых задач) предложены новые алгоритмы зеркального спуска аналогично для задач минимизации $\mu$-сильно выпуклых функционалов $f$ с неположительным, $\mu$-сильно выпуклым и липшицевым негладким функциональным ограничением $g$. Заметим, что техника рестартов метода для выпуклых задач с целью ускорения сходимости для сильно выпуклых задач восходит к 1980-м годам, см. \cite{bib_Nemirovski1983,bib_Nesterov1983}. Техника такого типа была использована в  \cite{bib_Juditsky_Nemirovski} для обоснования более высокой скорости сходимости метода зеркального спуска для сильно выпуклого целевого функционала в задачах без ограничений.

В-третьих, мы приводим ряд численных экспериментов, иллюстрирующих преимущества предложенных нами методов перед их аналогами. В частности, показано, что для задачи Ферма-Торричелли-Штейнера (целевой функционал удовлетворяет условию Липшица с константой 1) при наличии квадратичных ограничений предлагаемый нами метод может работать существенно быстрее, чем аналогичный адаптивный и также оптимальный для класса задач с липшицевым целевым функционалом с точки зрения оракульных оценок метод (\cite{bib_Adaptive}, п. 3.1). Также приведены расчёты, иллюстрирующие некоторые преимущества предлагаемых нами методов для сильно выпуклых задач.

Статья состоит из введения и 5 основных разделов. В разделе 2 мы приводим некоторые вспомогательных сведения, а также основные понятия для метода зеркального спуска. В разделе 3 мы описываем адаптивный алгоритм зеркального спуска (алгоритм \ref{algorithm1}) из (\cite{bib_Adaptive}, п. 3.3) и частично адаптивный алгоритм \ref{algorithm2FT} \cite{bib_Titov}. В разделе 4 мы доказываем оценки скорости сходимости данных методов и обосновываем их оптимальность на рассматриваемых классах задач при различных допущениях на уровень гладкости целевого функционала. Раздел 5 посвящён методам для задач минимизации сильно выпуклых функций с рестартами алгоритмов \ref{algorithm1} (алгоритм \ref{algorithm2}) и \ref{algorithm2FT} (алгоритм \ref{Alg2}), а также соответствующим теоретическим оценкам скорости сходимости. В последнем разделе мы приводим некоторые численные эксперименты, иллюстрирующие некоторые преимущества предлагаемых нами методов.

\section{Постановка задачи и основные понятия}

Пусть $(E,||\cdot||)$ --- конечномерное нормированное векторное пространство и $E^*$ --- сопряженное пространство к $E$ со стандартной нормой:
$$||y||_*=\max\limits_x\{\langle y,x\rangle,||x||\leq1\},$$
где $\langle y,x\rangle$ --- значение линейного непрерывного функционала $y$ в точке $x \in E$.

Пусть $X\subset E$ --- замкнутое выпуклое множество. Рассмотрим два выпуклых субдифференцируемых функционала $f$ и $g:X\rightarrow\mathbb{R}$. Также предположим, что функционал $g$ удовлетворяет условию Липшица относительно нормы $\| \cdot \|$, т. е. существует $ M_g> 0 $, такое, что
\begin{equation}\label{eq1}
|g(x)-g(y)|\leq M_g\lVert x-y \rVert
\end{equation}
для всяких $ x, y \in X $. Это означает, что в каждой точке $x \in X$ можно вычислить субградиент $\nabla g (x) $, причём $\|\nabla g(x)\|_* \leq M_g$. Напомним, что для дифференцируемого функционала $g$ субградиент $\nabla g (x)$ есть обычный градиент.

В настоящей работе будем рассматривать следующий тип задач  оптимизации:
\begin{equation}\label{eq2}
f(x) \rightarrow \min\limits_{x\in X},
\end{equation}
\begin{equation}
\label{problem_statement_g}
\hspace{0.3cm} g(x) \leq 0.
\end{equation}
если $f$ и $g$ удовлетворяют упомянутым предыдущим условиям. Сделаем предположение о разрешимости задачи \eqref{eq2} -- \eqref{problem_statement_g}.

Отметим, что часть результатов работы относятся к постановке задачи для $\mu$-сильно выпуклых субдифференцируемых функционалов $f$ и $g:X\rightarrow\mathbb{R}$, т.е. для произвольных $x, y \in X$ имеет место неравенство
\begin{equation} \label{def_strongly_convex}
f(y) \geq f(x) + \langle \nabla f(x), y-x \rangle + \frac{\mu}{2} \| y-x \|^2,
\end{equation}
и такое же неравенство верно для $g$ (с тем же параметром сильной выпуклости $\mu$).

Для дальнейших рассуждений нам также потребуются вспомогательные понятия (см., например, \cite{bib_Nemirovski}). Введём так называемую {\it прокс-функцию} $d : X \rightarrow \mathbb{R}$, обладающую свойством непрерывной дифференцируемости и $1$-сильной выпуклости относительно нормы $\lVert\cdot\rVert$, и предположим, что $\min\limits_{x\in X} d(x) = d(0).$  Будем полагать, что существует такая константа $\Theta_0 > 0$, что $d(x_{*}) \leq \Theta_0^2,$ где $x_*$  --- точное решение задачи (\ref{eq2})--(\ref{problem_statement_g}). Отметим, что если имеется множество решений $X_*$, то мы предполагаем, что для константы $\Theta_0$
$$\min\limits_{x_* \in X_*} d(x_*) \leq \Theta_0^2.$$
Для произвольных $x, y\in X$ рассмотрим соответствующую дивергенцию Брэгмана
$$V(x, y) = d(y) - d(x) - \langle \nabla d(x), y-x \rangle.$$

В зависимости от постановки конкретной задачи возможны различные подходы к определению прокс-структуры задачи и соответствующей дивергенции Брэгмана: евклидова, энтропийная и многие другие (см., например, \cite{bib_Nemirovski}). Стандартно определим оператор проектирования
$$
\mathrm{Mirr}_x (p) = \arg\min\limits_{u\in X} \big\{ \langle p, u \rangle + V(x, u) \big\} \;
\text{ для каждого }x\in X \text{ и }p\in E^*.
$$
Сделаем предположение о том, что оператор $\mathrm{Mirr}_x (p)$ легко вычислим.

Напомним одно известное утверждение (см., например \cite{bib_Nemirovski}).

\begin{lemma}\label{lem1}
Пусть  $f:X\rightarrow\mathbb{R}$ --- выпуклый субдифференцируемый функционал на выпуклом множестве $X$ и $z=Mirr_{y}(h\nabla f(y))$ для некоторого $y\in X$. Тогда для произвольных $x\in X$ и $h>0$ справедливо неравенство
\begin{equation} \label{equation7}
	h\langle\nabla f(y), y-x\rangle\leq\frac{h^2}{2}||\nabla f(y)||_*^2 + V(y,x) - V(z,x).
\end{equation}
\end{lemma}

\section{Адаптивный и частично адаптивный алгоритм зеркального спуска задач с выпуклыми функционалами}

Перейдём к описанию рассматриваемых методов \cite{bib_Adaptive,bib_Titov} для задач (\ref{eq2}) -- (\ref{problem_statement_g}).

Напомним следующий алгоритм адаптивного зеркального спуска для задач (\ref{eq2}) -- (\ref{problem_statement_g}) из (\cite{bib_Adaptive}, п. 3.3).
\begin{algorithm}[H]
	\caption{Адаптивный зеркальный спуск (нестандартные условия роста)}
	\label{algorithm1}
	\begin{algorithmic}[1]
		\REQUIRE $ \text{точность} \ \varepsilon>0; \text{начальная точка} \ x^0; \Theta_0; X; d(\cdot). $
		\STATE$I=:\emptyset$
		\STATE$N\leftarrow 0$
		\REPEAT
		\IF{$g(x^N) \leqslant \varepsilon$}
		\STATE $h_N\leftarrow\frac{\varepsilon}{\lVert \nabla f(x^N) \rVert_{*}}$
		\STATE$x^{N+1}\leftarrow Mirr_{x^N}(h_N\nabla f(x^N)) \;\; \text{("продуктивные шаги")}$
		\STATE $N\rightarrow I$
		\ELSE
		\STATE $(g(x^N)>\varepsilon)\rightarrow$
		\STATE $h_N\leftarrow\frac{\varepsilon}{\lVert \nabla g(x^N) \rVert_{*}^2}$
		\STATE $x^{N+1}\leftarrow Mirr_{x^N}(h_N\nabla g(x^N)) \;\; \text{("непродуктивные шаги")}$
		\ENDIF
		\STATE $N\leftarrow N+1$
		\UNTIL $\Theta_0^2 \leqslant \frac{\varepsilon^2}{2}\left(|I|+\sum\limits_{k\not\in I}\frac{1}{\lVert \nabla g(x^N) \rVert_{*}^2}\right)$
		\ENSURE $\bar{x}^N := \arg \min_{x^k, k\in I} f(x^k)$
	\end{algorithmic}
\end{algorithm}

Нам потребуется ввести для целевого функционала $f$ по аналогии с \cite{bib_Nesterov}, определим для некоторого субградиента $\nabla f(x)$ (мы допускаем, что в ходе работы метода можно использовать произвольный субградиент) в точке $y \in X$ следующую вспомогательную величину:

\begin{equation}
v_f(x, y)=\left\{
\begin{aligned}
&\left\langle\frac{\nabla f(x)}{\|\nabla f(x)\|_{*}},x-y\right\rangle, \quad &\nabla f(x) \ne 0\\
&0 &\nabla f(x) = 0\\
\end{aligned}
\right.,\quad x \in X.
\label{eq:vfDef}
\end{equation}

Для оценки скорости сходимости алгоритма \ref{algorithm1} в \cite{bib_Adaptive} получен следующий результат.

\begin{theorem}
Пусть верно неравенство \eqref {eq1} и известна константа $ \Theta_ {0}> 0 $ такова, что $d(x_*) \leq \Theta_{0}^{2}$. Если $\varepsilon > 0$ --- фиксированное число, то алгоритм \ref{algorithm1} работает не более
\begin{equation}\label{atsenka_alg1}
N=\left\lceil\frac{2\max\{1, M_g^2\}\Theta_0^2}{\varepsilon^2}\right\rceil
\end{equation}
итераций, причём после его остановки справедливо неравенство
	\begin{equation}\label{eq09}
	\min\limits_{k \in I} v_f(x^k,x_*)<\varepsilon.
	\end{equation}
	\label{theorem1}
\end{theorem}

Возможно \cite{bib_Titov} предложить также и частично адаптивный метод для задачи (\ref{eq2}) -- (\ref{problem_statement_g}).
Его отличие от алгоритма \ref{algorithm1} в том, что адаптивно выбирается шаг лишь на продуктивных итерациях и критерий остановки неадаптивен.

\begin{algorithm}
	\caption{Частично адаптивная версия Алгоритма \ref{algorithm1}}
	\label{algorithm2FT}
	\begin{algorithmic}[1]
		\REQUIRE $ \text{точность} \ \varepsilon>0; \text{начальная точка} \ x^0; \Theta_0; X; d(\cdot). $
		\STATE $x^0 = \text{argmin}_{x \in X} d(x) $
		\STATE$I=:\emptyset$
		\STATE$N\leftarrow 0$
		\REPEAT
		\IF{$g(x^N) \leqslant \varepsilon \rightarrow $ }
		\STATE $h_N\leftarrow\frac{\varepsilon}{M_{g} \cdot \lVert \nabla f(x^N) \rVert_{*}}$
		\STATE$x^{N+1}\leftarrow Mirr_{x^N}(h_N\nabla f(x^N)) \;\; \text{("продуктивные шаги")}$
		\STATE $N\rightarrow I$
		\ELSE
		\STATE $(g(x^N)>\varepsilon)\rightarrow$
		\STATE $h_N\leftarrow\frac{\varepsilon}{ M_g^{2}}$
		\STATE $x^{N+1}\leftarrow Mirr_{x^N}(h_N\nabla g(x^N)) \;\; \text{("непродуктивные шаги")}$
		\ENDIF
		\STATE $N\leftarrow N+1$
		\UNTIL $N \geq \left\lceil \dfrac{2M_{g}^{2} \Theta_0^2}{\varepsilon^2} \right\rceil $
		\ENSURE $\bar{x}^N := \arg \min_{x^k, k\in I} f(x^k)$
	\end{algorithmic}
\end{algorithm}

\newpage
Пусть $[N]=\{k\in\overline{0,N-1}\},\;J=[N]/I$, где $I$  набор индексов продуктивных шагов
\begin{equation}\label{equation4}
h_k=\frac{\varepsilon}{M_g||\nabla f(x^k)||_{*}},
\end{equation}
и $|I|$ --- количество "продуктивных шагов". Аналогично, для "непродуктивных шагов" из множества $J$ аналогичная переменная определяется следующим образом:
\begin{equation}\label{eq5}
h_k=\frac{\varepsilon}{M_g^2},
\end{equation}
и $|J|$ --- количество "непродуктивных шагов". Очевидно,
\begin{equation}\label{eq6}
|I|+|J|=N.
\end{equation}

Cправедлив следующий аналог теоремы \ref{theorem1} (см. также \cite{bib_Titov}).

\begin{theorem}\label{th1}
	Пусть $\varepsilon > 0$ --- фиксированное число и алгоритм \ref{algorithm2FT} работает
	\begin{equation}\label{eq8}
	N=\left\lceil\frac{2M_g^2\Theta_0^2}{\varepsilon^2}\right\rceil
	\end{equation}
	итераций. Тогда
	\begin{equation}\label{eq9}
	\min\limits_{k\in I} v_f(x^k,x_*)<\frac{\varepsilon}{M_g}.
	\end{equation}
\end{theorem}

\begin{proof}
	1) Для продуктивных шагов из \eqref{equation7}, \eqref{equation4} можно получить, что $$h_k\langle\nabla f(x^k), x^k-x\rangle\leq\frac{h_k^2}{2}||\nabla f(x^k)||_*^2+V(x^k,x)-V(x^{k+1},x).$$
	Принимая во внимание $\frac{h_k^2}{2}||\nabla f(x^k)||_*^2=\frac{\varepsilon^2}{2M_g^2}$, мы имеем
	\begin{equation}\label{eq001}
	h_k\langle\nabla f(x^k),x^k-x\rangle=\frac{\varepsilon}{M_g}\left\langle\frac{\nabla f(x^k)}{||\nabla f(x^k)||_*},\, x^k-x\right\rangle=\frac{\varepsilon}{M_g}v_f(x^k,x).
	\end{equation}
	
	2) Аналогично, для непродуктивных шагов $k\in J$:
	$$h_k(g(x^k)-g(x))\leq\frac{h_k^2}{2}||\nabla g(x^k)||_*^2+V(x^k,x)-V(x^{k+1},x).$$
	
	Используя \eqref{eq1} и $||\nabla g(x)||\leq M_g$, получаем
	\begin{equation}\label{eq002}
	h_k(g(x^k)-g(x))\leq\frac{\varepsilon^2}{2M_g^2}+V(x^k,x)-(x^{k+1},x).
	\end{equation}
	
	3) Из \eqref{eq001} и \eqref{eq002} для $x=x_*$, мы имеем
	$$
	\frac{\varepsilon}{M_g}\sum\limits_{k\in I}v_f(x^k,x_*)+\sum\limits_{k\in J}\frac{\varepsilon}{M_g^2}(g(x^k)-g(x_*))\leq
	$$
	\begin{equation}\label{eq10}
	\leq N\frac{\varepsilon^2}{2M_g^2}+
	\sum\limits_{k=0}^{N-1}(V(x^k,x_*)-V(x^{k+1},x_*)).
	\end{equation}
	
	Отметим, что для любого $k \in J$
	$$g(x^k)-g(x_*)\geq g(x^k)>\varepsilon$$
	и с учетом
	$$\sum\limits_{k=1}^N(V(x^k,x_*)-V(x^{k+1},x_*))\leq\Theta_0^2$$
	неравенство (\ref{eq10})  можно преобразовать следующим образом:
	$$\frac{\varepsilon}{M_g}\sum\limits_{k\in I}v_f(x^k,x_*)\leq N\frac{\varepsilon^2}{2M_g^2}+\Theta_0^2-\frac{\varepsilon^2}{M_g^2}|J|.$$
	
	С другой стороны,
	$$\sum\limits_{k\in I}v_f(x^k,x_*)\geq|I|\min\limits_{k\in I}v_f(x^k,x_*).$$
	
	Предположим, что
	\begin{equation}\label{eq11}
	\frac{\varepsilon^2}{2M_g^2}N\geq\Theta_0^2, \text{ или } N\geq\frac{2M_g\Theta_0^2}{\varepsilon^2}.
	\end{equation}
	Таким образом
	$$|I|\frac{\varepsilon}{M_g}\min v_f(x^k,x_*)<N\frac{\varepsilon^2}{2M_g^2}-\frac{\varepsilon^2}{M_g^2}|J|+\Theta_0^2\leq$$
	$$\leq\frac{N\varepsilon^2}{M_g^2}-\frac{\varepsilon^2}{M_g^2}|J|=\frac{\varepsilon^2}{M_g^2}|I|,$$
	откуда
	\begin{equation}\label{eq12}
	|I|\frac{\varepsilon}{M_g}\min v_f(x^k,x_*)<\frac{\varepsilon^2}{M_g^2}|I|\Rightarrow\min v_f(x^k,x_*)<\frac{\varepsilon}{M_g}.
	\end{equation}
	
	Чтобы закончить доказательство, мы должны показывать что  $|I|\neq 0$.
	Предположим наоборот, что $|I|=0\Rightarrow|J|=N$, т. е. все шаги непродуктивны, поэтому после использования
$$g(x^k)-g(x_*)\geq g(x^k)>\varepsilon$$
	мы можем видеть, что $$\sum\limits_{k=0}^{N-1} h_k(g(x^k)-g(x_*))\leq\sum\limits_{k=0}^{N-1}\frac{\varepsilon^2}{2M_g^2}+\Theta_0^2\leq
	N\frac{\varepsilon^2}{2M_g^2}+N\frac{\varepsilon^2}{2M_g^2}=N\frac{\varepsilon^2}{M_g^2}.
	$$
	Итак, $$\frac{\varepsilon}{M_g^2}\sum\limits_{k=0}^{N-1}(g(x^k)-g(x_*))\leq\frac{N\varepsilon^2}{M_g^2}$$
	и
	$$N\varepsilon<\sum\limits_{k=0}^{N-1}(g(x^k)-g(x_*))\leq N\varepsilon.$$
	
	Итак, мы получили противоречие и поэтому множество $I$ непусто.
\end{proof}

\begin{remark}
Поясним ситуацию, когда частично адаптивная версия алгоритма может оказаться более выгодной, чем адаптивная. Например, пусть имеется ситуация, когда нет возможности точного нахождения нормы (суб)градиента ограничения $\|\nabla g(x^k)\|_*$ для одного или нескольких непродуктивных шагов ($k \in J$), а известно лишь его некоторое приближение по норме: т.е. $\|\nabla g(x^k)\|_* = \alpha_k \pm \delta_k $, где $\delta_k$ --- точность приближения. По лемме \ref{lem1} на всяком непродуктивном шаге $x^k$ верно неравенство
	\begin{equation}\label{equation1}
	h_k \left( g(x^k) -g(x_*)\right) \leq \frac{h_k^2}{2}||\nabla g(x^k)||_*^2 + V(x^k,x_*) - V(x^{k+1},x_*).
	\end{equation}
Если $\alpha_k =0$ или $\alpha_k \rightarrow 0$, то мы не можем использовать неравенство \eqref{equation1}, поскольку это может привести к большой погрешности его правой части. В таком случае неадаптивный выбор шага $$h_k = \frac{\varepsilon}{M_g^2}$$
в алгоритме \ref{algorithm2FT} --- более подходящий вариант для решения задачи \eqref{eq2} -- \eqref{problem_statement_g}.
\end{remark}

\section{Оценки скорости сходимости рассмотренных методов и их оптимальность}

В данном разделе работы мы рассмотрим конкретные оценки скорости сходимости рассмотренных методов, которые обоснуют их оптимальность с точки зрения оракульных оценок (с точки зрения теории А.С. Немировского и Д.Б. Юдина). Точнее говоря ввиду липшицевости и, вообще говоря, негладкости функциональных ограничений для оптимальности метода с точки зрения нижних оракульных оценок этого достаточно показать \cite{bib_Nemirovski}, что для достижения требуемой точности $\varepsilon$ решения задачи \eqref{eq2}--\eqref{problem_statement_g} для каждого из рассмортренных в данном разделе статьи класса целевых функционалов достаточно $$O\left(\frac{1}{\varepsilon^2}\right)$$ итераций метода, предполагающих вычисление (суб)градиента целевого функционала или ограничения. Будем использовать следующее вспомогательное утверждение (см. например \cite{bib_Nesterov,bib_Nesterov2016}). Пусть $x_*$ ---  решение задачи \eqref{eq2} ---  \eqref{problem_statement_g}.

\begin{lemma}\label{Lemma_Nesterov}
	Введём следующую функцию:
	\begin{equation}\label{eq13}
	\omega(\tau)=\max\limits_{x\in X}\{f(x)-f(x_*):||x-x_*||\leq\tau\},
	\end{equation} где $\tau$ положительное число. Тогда для всякого $y \in X$
	\begin{equation}\label{eq_lemma}
	f(y) - f(x_*) \leq \omega(v_f(y,x_*)).
	\end{equation}
	\end{lemma}

Теперь мы можем показать (см. также доклад \cite{bib_Titov}), как с использованием предыдущего утверждения и теоремы \ref{th1}, можно оценить скорость сходимости алгоритма \ref{algorithm2FT}, если целевой функционал $f$ дифференцируем и его градиент удовлетворяет условию Липшица:
\begin{equation}\label{eqlipgrad}
||\nabla f(x)-\nabla f(y)||_*\leq L||x-y|| \quad \forall x,y\in X.
\end{equation}

Используя следующий известный факт
$$f(x)\leq f(x_*)+||\nabla f(x_*)||_*||x-x_*||+\frac{1}{2}L||x-x_*||^2,$$
мы можем получить
$$\min\limits_{k\in I}f(x^k)-f(x_*)\leq\min\limits_{k\in I} \left\{||\nabla f(x_*)||_*||x^k-x_*||+\frac{1}{2}L||x^k-x_*||^2\right\}.$$
Итак
$$f(x)-f(x_*)\leq ||\nabla f(x_*)||_*\frac{\varepsilon}{M_g}+\frac{L\varepsilon^2}{2M_g}.$$

Поэтому имеет место следующий результат \cite{bib_Titov}.
\begin{corollary}\label{col3}
Пусть $f$ дифференцируем на $X$ и верно \eqref{eqlipgrad}. Тогда после $$N=\left\lceil\frac{2M_g^2\Theta_0^2}{\varepsilon^2}\right\rceil$$ шагов работы алгоритма \ref{algorithm2FT} выполнена следующая оценка:  $$\min\limits_{0\leq k\leq N}f(x^k)-f(x_*)\leq ||\nabla f(x_*)||_*\frac{\varepsilon}{M_g}+\frac{L}{2}\frac{\varepsilon^2}{M_g^2}.$$
\end{corollary}

Мы можем применить наш метод к некоторому классу задач с негладкими целевыми функционалами специального типа \cite{bib_Titov}.

\begin{corollary}\label{col31}
	Предположим, что $f(x) = \max\limits_{i = \overline{1, m}} f_i(x)$, где $f_i$ дифференцируемы на каждой $x \in X$  и
	$$	||\nabla f_i(x)-\nabla f_i(y)||_*\leq L_i||x-y|| \quad \forall x,y\in X.	$$
	Тогда после $$N=\left\lceil\frac{2M_g^2\Theta_0^2}{\varepsilon^2}\right\rceil$$ шагов работы Алгоритма \ref{algorithm2FT} выполнена следующая оценка:
	$$\min\limits_{0\leq k\leq N}f(x^k)-f(x_*)\leq ||\nabla f(x_*)||_*\frac{\varepsilon}{M_g}+\frac{L}{2}\frac{\varepsilon^2}{M_g^2},	$$
	где $L = \max\limits_{i = \overline{1, m}} L_i$.
\end{corollary}

\begin{remark}
	Вообще $||\nabla f(x_*)||_* \neq 0$, поскольку мы рассматриваем некоторый класс условных задач.
\end{remark}

\begin{remark}\label{lipschits_objective}
Пусть целевой функционал $f:X\rightarrow\mathbb{R}$ удовлетворяет условию Липшица:
\begin{equation}\label{lipschits_condition}
|f(x)-f(y)|\leqslant M_f||x-y||\quad \forall x,y\in X.
\end{equation}
Итак
$$f(x) \leq f(x_*)+M_f||x - x_*||,$$
мы можем получить
$$\min\limits_{k\in I}f(x^k)-f(x_*)\leq\min\limits_{k\in I} \left\{M_f ||x^k - x_*||\right\}.$$
Итак, комбинируя утверждения теоремы \ref{theorem1} и леммы \ref{Lemma_Nesterov}, мы можем гарантировать после остановки алгоритма 1
выполнение неравенства
$$f(x)-f(x_*)\leq M_f \varepsilon,$$
и аналогично из теоремы \ref{th1} для алгоритма 2:
$$f(x)-f(x_*)\leq \frac{M_f}{M_g}\varepsilon.$$
Поэтому имеет место следующий результат.
\begin{corollary}\label{cor_lipschits1}
	Если $f$ удовлетворяет условию Липшица \eqref{lipschits_condition} на $X$. Тогда
	\begin{itemize}
		\item после $$N=\left\lceil\frac{2\max\{1, M_g^2\}\Theta_0^2}{\varepsilon^2}\right\rceil$$ шагов работы алгоритма \ref{algorithm1}, выполнена следующая оценка:
		$$\min\limits_{1\leq k\leq N}f(x^k)-f(x_*)\leq M_f \varepsilon;$$
		\item после $$N=\left\lceil\frac{2M_g^2\Theta_0^2}{\varepsilon^2}\right\rceil$$ шагов работы алгоритма \ref{algorithm2FT}, выполнена следующая оценка:
		  $$\min\limits_{1\leq k\leq N}f(x^k)-f(x_*)\leq \frac{M_f}{M_g}\varepsilon.$$
	\end{itemize}
\end{corollary}
\end{remark}

\begin{remark}\label{lips_Hessian_objective}
	Пусть целевой функционал $f:X\rightarrow\mathbb{R}$ дважды дифференцируем на $X$ и имеет липшицев гессиан, т.е. справедливо следующее неравенство
	\begin{equation}\label{lips_Hessian_condition}
	||\nabla^{2} f(x)-\nabla^{2} f(y)||_*\leq L ||x-y|| \quad \forall x,y\in X.
	\end{equation}
	Используя следующее неравенство (см. \cite{bib_Nesterov}, лемма 1.2.4 )
	$$ |f(x) - f(x_*) - \left\langle \nabla f(x_*), x-x_* \right\rangle -\dfrac{1}{2} \left\langle \nabla^2 f(x_*)(x-x_*), x-x_* \right\rangle  | \leq  \frac{L}{6}||x-x_*||^3,$$
	мы можем видеть, что
	$$ f(x) \leq f(x_*) + ||\nabla f(x_*) ||\cdot||x-x_*|| + \dfrac{1}{2}||\nabla^2 f(x_*)(x-x_*)||\cdot ||x-x_*||+\dfrac{L}{6}||x-x_*||^3
	$$
Итак
	$$ f(x) \leq f(x_*) + ||\nabla f(x_*) ||\cdot||x-x_*|| + \dfrac{1}{2} ||\nabla^2 f(x_*)||_{Fro} \cdot ||x-x_*||^2+\dfrac{L}{6}||x-x_*||^3
	$$
	где $||A||_{Fro}= tr(A^{T}A)$ норма Фробениуса матрицы $A \in \mathbb{R}^{m\times n}$. Тогда
	$$\min\limits_{k\in I}f(x^k)-f(x_*)\leq\min\limits_{k\in I} \left\{||\nabla f(x_*)||\cdot||x^k-x_*||+\frac{1}{2}||\nabla^2 f(x_*)||_{Fro}\cdot||x^k-x_*||^2 + \dfrac{L}{6}||x^k-x_*||^3\right\}.$$

Итак, комбинируя утверждение теоремы \ref{theorem1} и леммы \ref{Lemma_Nesterov}, возможно получить
	$$f(x)-f(x_*)\leq ||\nabla f(x_*)||_*\cdot\varepsilon + \frac{1}{2}||\nabla^2 f(x_*)||_{Fro} \cdot\varepsilon^2 + \frac{L}{6}\varepsilon^3,$$
а также аналогично из теоремы \ref{th1}
    $$f(x)-f(x_*)\leq ||\nabla f(x_*)||_*\cdot \frac{\varepsilon}{M_g} + \frac{1}{2}||\nabla^2 f(x_*)||_{Fro} \cdot\frac{\varepsilon^2}{M_g^2} + \frac{L}{6}\frac{\varepsilon^3}{M_g^3}.$$

Поэтому имеет место следующий результат.
	\begin{corollary}\label{cor_lips_Hessian}
Пусть $f$ дважды дифференцируем на $X$ и имеет липшицев гессиан, т.е. верно \eqref{lips_Hessian_condition}. Тогда
		\begin{itemize}
			\item после $$N=\left\lceil\frac{2\max\{1, M_g^2\}\Theta_0^2}{\varepsilon^2}\right\rceil$$ шагов работы алгоритма \ref{algorithm1} выполнена следующая оценка:
			$$\min\limits_{1\leq k\leq N}f(x^k)-f(x_*)\leq ||\nabla f(x_*)||_*\cdot\varepsilon + \frac{1}{2}||\nabla^2 f(x_*)||_{Fro} \cdot\varepsilon^2 + \frac{L}{6}\varepsilon^3;$$
			\item после $$N=\left\lceil\frac{2M_g^2\Theta_0^2}{\varepsilon^2}\right\rceil$$ шагов работы алгоритма \ref{algorithm2FT} выполнена следующая оценка:
			$$\min\limits_{1\leq k\leq N}f(x^k)-f(x_*)\leq ||\nabla f(x_*)||_*\cdot \frac{\varepsilon}{M_g} + \frac{1}{2}||\nabla^2 f(x_*)||_{Fro} \cdot\frac{\varepsilon^2}{M_g^2} + \frac{L}{6}\frac{\varepsilon^3}{M_g^3}.$$
		\end{itemize}
	\end{corollary}
	
Мы можем применить наши методы к некоторому классу задач с негладкими целевыми функционалами.
	
	\begin{corollary}\label{cor_lips_Hessian_max}
		Предположим, что $f(x) = \max\limits_{i = \overline{1, m}} f_i(x)$, где $f_i$ дважды дифференцируемы в каждой точке $x \in X$  и
		$$||\nabla^{2} f(x)-\nabla^{2} f(y)||_*\leq L_i ||x-y|| \quad \forall x,y\in X.$$
		Тогда
		\begin{itemize}
			\item после $$N=\left\lceil\frac{2\max\{1, M_g^2\}\Theta_0^2}{\varepsilon^2}\right\rceil$$ шагов работы алгоритма \ref{algorithm1} выполнена следующая оценка:
			$$\min\limits_{1\leq k\leq N}f(x^k)-f(x_*)\leq ||\nabla f(x_*)||_*\cdot\varepsilon + \frac{1}{2}||\nabla^2 f(x_*)||_{Fro} \cdot\varepsilon^2 + \frac{L}{6}\varepsilon^3,$$
			где $L=\max\limits_{i = \overline{1, m}} L_i$;
			\item после $$N=\left\lceil\frac{2M_g^2\Theta_0^2}{\varepsilon^2}\right\rceil$$ шагов работы алгоритма \ref{algorithm2FT} выполнена следующая оценка:
			$$\min\limits_{1\leq k\leq N}f(x^k)-f(x_*)\leq ||\nabla f(x_*)||_*\cdot \frac{\varepsilon}{M_g} + \frac{1}{2}||\nabla^2 f(x_*)||_{Fro} \cdot\frac{\varepsilon^2}{M_g^2} + \frac{L}{6}\frac{\varepsilon^3}{M_g^3},$$
			где $L=\max\limits_{i = \overline{1, m}} L_i$.
		\end{itemize}
	\end{corollary}
\end{remark}

\section{Об ускорении рассматриваемых методов зеркального спуска для сильно выпуклых задач}

В этом разделе работы мы рассмотрим задачу
\begin{equation}\label{eq14}
f(x)\rightarrow\min,\;\;g(x)\leq 0,\;\;x\in X
\end{equation}
с предположениями \eqref{eq1}, а также сильной выпуклости $f$ и $g$ с одинаковым параметром $\mu > 0$. Мы также слегка модифицируем предположения на прокс-функцию $d(x)$. А именно, предположим, что $0 = \arg \min \limits_{x \in X} d(x)$ и что $d$ ограничено на единичном шаре в выбранной норме $ \| \cdot \| $, т. е.
\begin{equation}
d(x) \leq \Theta_{0}^{2}, \quad \forall x\in X : \|x \| \leq 1,
\label{eq:dUpBound}
\end{equation}
Наконец, мы допускаем, что нам дана начальная точка $x^0 \in X$ и число $R_0 >0$ такое, что $\| x_0 - x_* \|^2 \leq R_0^2$. Для построения метода решения задачи \eqref{eq14} при заданных предположениях мы используем идею рестартов (перезапусков) алгоритма \ref{algorithm1} и алгоритма \ref{algorithm2FT}. Рассмотрим вспомогательное утверждение (см., например \cite{bayandina2018primal-dual}).

\begin{lemma}\label{lem3}
	Если $f$ и $g$ --- $\mu$-сильно выпуклые функционалы относительно нормы $\|\cdot\|$ на $X$, $x_{\ast} = arg\min\limits_{x \in X} f(x)$, $g(x)\leq 0$ ($\forall x \in X$) и для некоторых $\varepsilon_{f}>0$, а также $\varepsilon_{g}>0$ верно:
	\begin{equation}\label{eq15}
	f(x)-f(x_{\ast})\leq \varepsilon_{f},\;\;g(x)\leq\varepsilon_{g}.
	\end{equation}
Тогда
	\begin{equation}\label{eq16}
	\frac{\mu}{2}\|x-x_{\ast}\|^{2}\leq\max\{\varepsilon_{f},\varepsilon_{g}\}.
	\end{equation}
\end{lemma}
Предположим, что $f(x) = \max\limits_{i = \overline{1, m}} f_i(x)$, где $f_i$ дифференцируемы во всякой точке $x \in X$ и имеют с липшицев градиент, т. е. существуют $ L_i> 0 $ такие, что
\begin{equation}\label{condition_Lip}
\lVert \nabla f_i(x)-\nabla f_i(y) \rVert_*\leq L_i\lVert x-y \rVert \quad \forall x,y\in X.
\end{equation}
Рассмотрим функцию $\tau: \mathbb{R}^{+}\rightarrow\mathbb{R}^{+}$:
\begin{equation}
\tau(\delta)=\max\left\{\delta\|\nabla f(x_{\ast})\|_{\ast}+\frac{\delta^{2}L}{2} , \; \delta \right\},
\end{equation}
где
$$
L:= \max\limits_{i = \overline{1, m}}\{L_i\}.
$$

Ясно, что функция $\tau $ возрастает и поэтому для всякого $\varepsilon>0$ существует
$$\hat{\varphi}(\varepsilon)>0:\;\;\tau(\hat{\varphi}(\varepsilon))=\varepsilon.$$

Рассмотрим следующий адаптивный aлгоритм \ref{algorithm2} для задачи \eqref{eq14}.

\begin{algorithm}[H]
	\caption{Адаптивный алгоритм зеркального спуска для сильно выпуклых
функционалов.}
	\label{algorithm2}
	\begin{algorithmic}[1]
		\REQUIRE $ \text{точность} \ \varepsilon>0; \text{начальная точка} \ x_0;$  $ \Theta_0 \  \text{s.t.} \ d(x)\leq\Theta_0^2 \quad \forall x\in X: \|x \| \leq 1;$ $X; d(\cdot);$  $ \text{параметр сильно выпуклости} \ \mu; R_0 \ \text{s.t.} \ \| x^0 - x_* \|^2 \leq R_0^2.$
		\STATE Set $d_0(x) = d\left(\frac{x-x^0}{R_0}\right)$.
		\STATE Set $p=1.$
		\REPEAT
		\STATE Set $R_p^2 = R_0^2 \cdot 2^{-p}.$
		\STATE Set $\varepsilon_p = \frac{\mu R_p^2}{2}.$
		\STATE Set $x^p$ --- выход алгоритма \ref{algorithm1} с точностью $\varepsilon_p$, прокс-функцией $d_{p-1}(\cdot)$ и $\Theta_0^2.$
		\STATE $d_p(x) \gets d\left(\frac{x - x^p}{R_p}\right)$.
		\STATE Set $p = p + 1.$
		\UNTIL $p>\log_2 \frac{\mu R_0^2}{2\varepsilon}.$
		\ENSURE $x^p.$
	\end{algorithmic}
\end{algorithm}

\begin{theorem}
Пусть $f$ имеет липшицев градиент, удовлетворяющий \eqref{condition_Lip}. Если $f$ и $g$ --- $\mu$-сильно выпуклые функционалы на $X\subset\mathbb{R}^{n}$  и $d(x)\leq \Theta^{2}_{0}$  для всех $x\in X,$ таких, что $\|x\|\leq1$. Пусть начальное приближение $x^{0}\in X$ и число $R_{0}>0$ заданы так, что
$$\|x^{0}-x_{\ast}\|^{2}\leq R^{2}_{0}.$$
Тогда для $\displaystyle{\widehat{p}=\left\lceil\log_{2}\frac{\mu R_{0}^{2}}{2\varepsilon}\right\rceil}$ выход $x_{\widehat{p}}$ есть $\varepsilon$-решение задачи \eqref{eq14} (т. е. $f(x^{\widehat{p}}) - f(x_*) < \varepsilon$ и $g(x^{\widehat{p}}) < \varepsilon$), где
$$
	\|x^{\widehat{p}}-x_{*}\|^{2}\leq\frac{2\varepsilon}{\mu}.
$$
При этом, количество итераций алгоритма \ref{algorithm1} не более
$$
	\widehat{p}+\sum_{p=1}^{\widehat{p}}\frac{2\Theta^{2}_{0}\max\{1, M_g^{2}\}}{\hat{\varphi}^{2}(\varepsilon_{p})},\;\;\text{где}\;\;\varepsilon_{p}=\frac{\mu R^{2}_{0}}{2^{p+1}}
	$$
	 итераций.
	
	\label{theorem2}
\end{theorem}

\begin{proof}
	Функция $d_{p}(x)=d\left( \dfrac{x-x^{p}}{R_{p}} \right)$, которая определена в алгоритме \ref{algorithm2}, является $1$-сильно выпуклой функцией относительно нормы $\dfrac{\|.\|}{R_{p}}$, для всех $p\geq 0$. Математической индукцией мы покажем, что
	 $$ \|x^{p}-x_{*}\|^{2} \leq R_{p}^{2} \quad \forall p \geq 0. $$
	Для $p=0$ это утверждение очевидно из-за выбора $x^0$ и $R_0$.
	Предположим, что для некоторого $p$, у нас $\|x^{p}-x_{*}\|^{2} \leq R_{p}^{2}$, и давайте докажем, что
	$\|x^{p+1}-x_{*}\|^{2} \leq R_{p+1}^{2}$. Имеем $\|x^{p}-x_{*}\|^{2} \leq R_{p}^{2}$. Докажем, что
	$\|x^{p+1}-x_{*}\|^{2} \leq R_{p+1}^{2}$.
	У нас $d_{p}(x_*) \leq \Theta_{0}^{2}$, таким образом, по теореме \ref{theorem1}, на $(p+1)$-м рестарте после не более чем
	$$
	N_{p+1}=\left\lceil\dfrac{2 \Theta_{0}^{2} \max\{1, M_{g}^{2}\}}{\hat{\varphi}^{2}(\varepsilon_{p+1})}\right\rceil
	$$
	итераций алгоритма \ref{algorithm1}, следующие неравенства верны для $x^{p+1}= \bar{x}^{N_{p+1}}$:
	$$
	f(x^{p+1})-f(x_{*}) \leq \varepsilon_{p+1}, \quad g(x^{p+1})\leq \varepsilon_{p+1} \quad \text{for} \quad  \varepsilon_{p+1}=\dfrac{\mu R_{p+1}^{2}}{2}.
	$$
	Тогда, согласно лемме \ref{lem3}
	$$\|x^{p+1}-x_{*}\|^{2} \leq \dfrac{2 \varepsilon_{p+1}}{\mu} = R_{p+1}^{2}.$$
	Итак, для всех $p \geq 0$ мы доказали, что
	$$
	\|x^{p}-x_{*}\|^{2} \leq R_{p}^{2} = \dfrac{R_{0}^{2}}{2^p},
	\quad f(x^p)-f(x_*)\leq \dfrac{\mu R_{0}^{2}}{2^{p+1}},
	\quad g(x^p) \leq \dfrac{\mu R_{0}^{2}}{2^{p+1}}.
	$$
	и так, для
	$p=\displaystyle{\widehat{p}=\left\lceil\log_{2}\frac{\mu R_{0}^{2}}{2\varepsilon}\right\rceil}$, $x_p$ это $\varepsilon$-решение задачи \eqref{eq14} и справедливо следующее соотношение
	$$	\|x^{p}-x_{*}\|^{2} \leq R_{p}^{2}= \dfrac{R_{0}^{2}}{2^{p}} \leq \dfrac{2 \varepsilon}{\mu}.	$$
	Итак, пусть $K$ обозначим общее число итераций алгоритма \ref{algorithm1}, и $N_{p}$ к общему числу итераций алгоритма \ref{algorithm1} на $p$-м рестарте. Поскольку функция
	$\tau: \mathbb{R}^{+}\rightarrow\mathbb{R}^{+}$,
	возрастает и для каждого $\varepsilon > 0$ существует
	$\hat{\varphi}(\varepsilon)>0:\;\tau(\hat{\varphi}(\varepsilon))=\varepsilon$.
	Поэтому мы имеем
	$$	K=\sum_{p=1}^{\widehat{p}} N_{p}=\sum_{p=1}^{\widehat{p}} \left\lceil \frac{2\Theta^{2}_{0}\max\{1, M_g^{2}\}}{\hat{\varphi}^{2}(\varepsilon_{p})} \right\rceil \leq \widehat{p}+\sum_{p=1}^{\widehat{p}}\frac{2\Theta^{2}_{0}\max
		\{1, M_g^{2}\}}{\hat{\varphi}^{2}(\varepsilon_{p})}. $$
\end{proof}

\begin{remark} Предыдущую оценку количества итераций работы алгоритма \ref{algorithm1} можно несколько конкретизировать в случае $\varepsilon < 1$. В этом случае при всяком $\delta < 1$ имеем $\tau(\delta) \leq C\delta$ для некоторой константы $C$. Поэтому можно считать, что $\hat{\varphi}(\varepsilon) = \widehat{C} \cdot \varepsilon$ для соответствующей константы $\widehat{C} > 0$.
Это означает, что на $p+1$-м рестарте алгоритма \ref{algorithm1} после не более, чем
		\begin{equation}
		k_{p+1} = \left\lceil\frac{\Omega \max\{1, M_g^2\} R_{p}^2}{\varepsilon_{p+1}^2}\right\rceil
		\label{eq:kpp1Est}
		\end{equation}
итераций работы алгоритма \ref{algorithm1}, выход $x_{p+1}$ гарантированно удовлетворяет неравенству
		$$
		f(x^{p+1})-f(x_*) \leq \widehat{C} \cdot \varepsilon_{p+1}, \quad g(x^{p+1}) \leq \varepsilon_{p+1},
		$$
		где $\varepsilon_{p+1} = \frac{\mu R_{p+1}^2}{2}$.
Тогда по лемме \ref{lem3},
		$$
		\| x^{p+1} - x_* \|^2 \leq \frac{2 \max\{1, \widehat{C}\}\varepsilon_{p+1}}{\mu} = \max\{1, \widehat{C}\} \cdot R_{p+1}^2.
		$$
Таким образом, всех $p\geq 0$,
$$\|x^p-x_*\|^2\leq \max\{1, \widehat{C}\} \cdot R_p^2 =  \max\{1, \widehat{C}\} \cdot R_0^2 \cdot 2^{-p}.$$

В то же время мы имеем для всяких $p\geq 1$ имеют место неравенства:
		$$
		f(x^{p})-f(x_*) \leq \max\{1, \widehat{C}\} \cdot \frac{\mu R_{0}^2}{2} \cdot 2^{-p}, \quad g(x_{p}) \leq \max\{1, \widehat{C}\} \cdot \frac{\mu R_{0}^2}{2} \cdot 2^{-p}.
		$$
		Таким образом, если  $p > \log_2 \frac{\mu R_0^2}{2 \varepsilon}$, то $x_p$ будет $\max\{1, \widehat{C}\} \cdot \varepsilon$-решением для поставленной задачи, причём:
		$$
		\|x^p-x_*\|^2\leq \max\{1, \widehat{C}\} \cdot R_0^2 \cdot 2^{-p} \leq \frac{2 \varepsilon}{\mu}.
		$$
		
	Оценим теперь общее число $N$ итераций алгоритма \ref{algorithm1}. Пусть $\hat{p} =  \left\lceil \log_2 \frac{\mu R_0^2}{2\varepsilon}\right\rceil$. Тогда согласно \eqref{eq:kpp1Est}, мы имеем с точностью до умножения на константу:
		\begin{align}
		N &= \sum_{p=1}^{\hat{p}} k_p \leq \sum_{p=1}^{\hat{p}} \left(1 + \frac{2\Theta_0^2 \max\{1, M_g^2\} R_{p}^2}{\varepsilon_{p+1}^2}\right) = \sum_{p=1}^{\hat{p}} \left(1 + \frac{32\Theta_0^2 \max\{1, M_g^2\}2^p}{\mu^2 R_0^2}\right) \notag \\
		& \leq \hat{p} + \frac{64 \Theta_0^2 \max\{1, M_g^2\}2^{ \hat{p} }}{\mu^2 R_0^2} \leq \hat{p} + \frac{64 \Theta_0^2 \max\{1, M_g^2\} }{\mu \varepsilon}. \notag
		\end{align}
\end{remark}

\begin{remark}
Вообще говоря, $\varphi(\varepsilon)$ зависит от $\|\nabla f(x_{\ast})\|_{\ast}$ и константа Липшица $L$ для $\nabla f$.
	Если $\|\nabla f(x_{\ast})\|_{\ast} < M_g$, тогда $\varphi(\varepsilon) = \varepsilon$ для небольших достаточно
	$\varepsilon$:
	$$ \varepsilon < \frac{2(M_g - \|\nabla f(x_{\ast})\|_{\ast})}{L}. $$
	Для другого случая  ($\|\nabla f(x_{\ast})\|_{\ast} > M_g$)  у нас $\forall \varepsilon > 0$:
	$$	\varphi(\varepsilon) = \frac{\sqrt{\|\nabla f(x_{\ast})\|_{\ast}^2 + 2\varepsilon L} - \|\nabla f(x_{\ast})\|_{\ast}}{L}.	 $$
\end{remark}

Рассмотрим также следующую частично адаптивную версию алгоритма \ref{Alg2} для задачи \eqref{eq14} \cite{bib_Titov}.

\begin{algorithm}[h!]
	\caption{ Частично адаптивный алгоритм зеркального спуска для сильно выпуклых функционалов}
	\label{Alg2}
	\begin{algorithmic}[1]
		\REQUIRE $ \text{точность} \ \varepsilon>0; \text{начальная точка} \ x^0;$  $ \Theta_0 \  \text{s.t.} \ d(x)\leq\Theta_0^2 \quad \forall x\in X: \|x \| \leq 1;$ $X; d(\cdot);$  $ \text{параметр сильно выпуклости} \ \mu; R_0 \ \text{s.t.} \ \| x^0 - x_* \|^2 \leq R_0^2.$
		\STATE  Set $d_0(x) = d\left(\frac{x-x^0}{R_0}\right)$.
		\STATE Set $p=1$.
		\REPEAT
		\STATE Set $R_p^2 = R_0^2 \cdot 2^{-p}$.
		\STATE Set $\varepsilon_p = \frac{\mu R_p^2}{2}$.
		\STATE Set $x^p$ --- выход алгоритма \ref{algorithm2FT} с точностью $\varepsilon_p$, прокс-функцией $d_{p-1}(\cdot)$ и $\Theta_0^2.$
		\STATE $d_p(x) \gets d\left(\frac{x - x^p}{R_p}\right)$.
		\STATE Set $p = p + 1$.
		\UNTIL{$p>\log_2 \frac{\mu R_0^2}{2\varepsilon}$.}
		\ENSURE $x^p$.
	\end{algorithmic}
\end{algorithm}

В условиях следствия \ref{col31} после остановки алгоритма \ref{Alg2} будут верными неравенства \eqref{eq15}
для $$\displaystyle{\varepsilon_{f}=\frac{\varepsilon}{M_g}\|\nabla f(x_{\ast})\|_{\ast}+\frac{\varepsilon^{2}L}{2M_g^2}}$$
и $\varepsilon_{g}=\varepsilon$. Рассмотрим функцию $\tau: \mathbb{R}^{+}\rightarrow\mathbb{R}^{+}$:
$$\tau(\delta)=\max\left\{\delta\|\nabla f(x_{\ast})\|_{\ast}+\frac{\delta^{2}L}{2}; \; \delta M_g\right\}.$$
Ясно, что функция $\tau$ возрастает и поэтому для каждого  $\varepsilon>0$ существует
$$\varphi(\varepsilon)>0:\;\;\tau(\varphi(\varepsilon))=\varepsilon.$$

Справедлива следующая

\begin{theorem}\label{th2}
Пусть $f$ и $g$ удовлетворяют условиям следствия \ref{col31}. Если $f$ и $g$ --- $\mu$-сильно выпуклые функционалы на $X\subset\mathbb{R}^{n}$ и $d(x)\leq \Theta^{2}_{0}\;\;\forall\,x\in X,\;\|x\|\leq1$. Пусть начальное приближение $x^{0}\in X$ и число $R_{0}>0$ заданы так, что $\|x^{0}-x_{\ast}\|^{2}\leq R^{2}_{0}$. Тогда для $\displaystyle{\widehat{p}=\left\lceil\log_{2}\frac{\mu R_{0}^{2}}{2\varepsilon}\right\rceil}$ выход $x^{\widehat{p}}$ есть $\varepsilon$-решение задачи \eqref{eq14}, где $$\|x^{\widehat{p}}-x_{\ast}\|^{2}\leq\frac{2\varepsilon}{\mu}.$$
При этом общее количество итераций алгоритма \ref{algorithm2FT} не превышает
 $$\widehat{p}+\sum_{p=1}^{\widehat{p}}\frac{2\Theta^{2}_{0}Mg^{2}}{\varphi^{2}(\varepsilon_{p})},\;\;\text{где}\;\;\varepsilon_{p}=\frac{\mu R^{2}_{0}}{2^{p+1}}.$$
\end{theorem}
\begin{proof}
	Функция $d_{p}(x)$ ($p=0,1,2,\ldots$) $1$-сильно выпукла относительно нормы $\displaystyle\frac{\|\cdot\|}{R_{p}}$, для всех $p\geq 0$. Методом математической индукции покажем, что
	$$\|x^{\widehat{p}}-x_{\ast}\|^{2}\leq R_{p}\;\;\forall p\geq0.$$
	Для $p=0$ это утверждение очевидно в силу выбора $x_0$ и $R_0$.
	Предположим, что для некоторого $p$: $\|x^{p}-x_{\ast}\|^{2}\leq R_{p}^{2}$. Докажем, что
	$\|x^{p+1}-x_{\ast}\|^{2}\leq R_{p+1}^{2}$. У нас  $d_{p}(x_{\ast})\leq\Theta^{2}_{0}$, и на $(p+1)$-м рестарте после не более чем
	$$\displaystyle\left\lceil\frac{2\Theta^{2}_{0}M_g^{2}}{\varphi^{2}(\varepsilon_{p+1})}\right\rceil$$
	итераций алгоритма \ref{algorithm2FT} будут выполняться следующие неравенства:
	$$f(x^{p+1})-f(x_{\ast})\leq\varepsilon_{p+1},\;\;\;g(x^{p+1})\leq\varepsilon_{p+1}\;\text{ для }\;\varepsilon_{p+1}=\frac{\mu R^{2}_{p+1}}{2}.$$
	Тогда, согласно лемме \ref{lem3}
	$$\|x^{p+1}-x_{\ast}\|^{2}\leq\frac{2\varepsilon_{p+1}}{\mu}=R^{2}_{p+1}.$$
	Итак, для произвольного $p\geq0$
	$$\|x^{p}-x_{\ast}\|^{2}\leq R^{2}_{p}=\frac{R^{2}_{0}}{2^{p}},\;\;\;f(x^{p})-f(x_{\ast})\leq\frac{\mu R^{2}_{0}}{2}2^{-p},\;\;\;g(x_{p})\leq\frac{\mu R^{2}_{0}}{2}2^{-p}.$$
	Для $\displaystyle{p=\widehat{p}=\left\lceil\log_{2}\frac{\mu R_{0}^{2}}{2\varepsilon}\right\rceil}$ верное следующее соотношение:
	$$\|x^{p}-x_{\ast}\|^{2}\leq R^{2}_{p}=R^{2}_{0}\cdot2^{-p}\leq\frac{2\varepsilon}{\mu}.$$

Остается лишь заметить, что количество итераций работы алгоритма \ref{algorithm2FT} не превосходит
	 $$\sum_{p=1}^{\widehat{p}}\left\lceil\frac{2\Theta^{2}_{0}M_g^{2}}{\varphi^{2}(\varepsilon_{p+1})}\right\rceil\leq\widehat{p}+\sum_{p=1}^{\widehat{p}}\frac{2\Theta^{2}_{0}M_g^{2}}{\varphi^{2}(\varepsilon_{p+1})}.$$
\end{proof}

\begin{remark}
По аналогии с рассуждениями замечания 5, при $\varepsilon < 1$ с точностью до умножения на константу можно уточнить верхнюю оценку количества итераций \ref{algorithm2FT}:
\begin{align}
		N = \hat{p} + \frac{64\Theta_0^2 M_g^2 \cdot 2^{\hat{p}}}{\mu^2 R_0^2} \leq \hat{p} + \frac{64\Theta_0^2 \cdot M_g^2}{\mu \varepsilon}. \notag
\end{align}
\end{remark}

\begin{remark}
Обратив внимание на следствия \ref{cor_lipschits1} и \ref{cor_lips_Hessian_max}, нетрудно понять, что при условии $\varepsilon < 1$ утверждения замечаний 5 и 7 нетрудно перенести и на случаи, когда целевой функционал $f$ удовлетворяет условию Липшица или условию Липшица для гессиана $f$.
\end{remark}

\section{Численные эксперименты}

\subsection{Сравнение скорости работы методов для задачи Ферма-Торричелли-Штейнера с ограничениями.} Отметим, что в (\cite{bib_Adaptive}, п. 3.1) предложен также следующий адаптивный метод, оптимальный с точки зрения нижних оракульных оценок в случае задач с липшицевым целевым функционалом.

\begin{algorithm}
\caption{Адаптивный зеркальный спуск (липшицев целевой функционал)}
\label{alg_lip}
\begin{algorithmic}[1]
\REQUIRE $\varepsilon>0,\Theta_0: \,d(x_*)\leqslant\Theta_0^2$
\STATE $x^0=argmin_{x\in X}\,d(x)$
\STATE $I=:\emptyset$
\STATE $N\leftarrow0$
\REPEAT
    \IF{$g(x^N)\leqslant\varepsilon$}
        \STATE $M_N=||\nabla f(x^N)||_*$, $h_N=\frac{\varepsilon}{M_N^2}$
        \STATE $x^{N+1}=Mirr_{x^N}(h_N\nabla f(x^N))\;\text{// \emph{"продуктивные шаги"}}$
        \STATE $N\rightarrow I$
    \ELSE
        \STATE $M_N=||\nabla g(x^N)||_*$, $h_N=\frac{\varepsilon}{M_N^2}$
        \STATE $x^{N+1}=Mirr_{x^N}(h_N\nabla g(x^N))\;\text{// \emph{"непродуктивные шаги"}}$
    \ENDIF
    \STATE $N\leftarrow N+1$
\UNTIL{$\sum\limits_{j=0}^{N-1}\frac{1}{M_j^2}\geqslant2\frac{\Theta_0^2}{\varepsilon^2}$}
\ENSURE $\bar{x}^N:=\frac{\sum\limits_{k\in I}x^kh_k}{\sum\limits_{k\in I}h_k}$
\end{algorithmic}
\end{algorithm}

\newpage

В настоящей работе мы рассматриваем альтернативный метод (алгоритм 1), оптимальность которого уже удаётся установить для условных задач с более широким классом целевых функционалов (имеющих липшицев градиент или липшицев гессиан). Но оказывается, что и в случае липшицевого целевого функционала, когда применим алгоритм 5, алгоритм 1 может работать быстрее. В качестве примера приведём расчёты для известной {\it задачи Ферма-Торричелли-Штейнера с ограничениями}.

\textbf{Задача.} \textit{Для заданных точек $A_k=(a_{1k},a_{2k},\ldots,a_{nk},)$ в n-мерном евклидовом пространстве $\mathbb{R}^n$ необходимо найти такую точку $X=(x_1,x_2,\ldots,x_n)$, чтобы целевая функция
$$f(x):=\sum_{k=1}^n\sqrt{(x_1-a_{1k})^2+(x_2-a_{2k})^2+\ldots+(x_n-a_{nk})^2}$$
принимала наименьшее значение на множестве $X$, которое задаётся несколькими ограничениями:}
$$g_1((x_1,\ldots,x_{10}))= 2x_1^2+x_2^2+\ldots+x_{10}^2-1 \leqslant 0,$$
$$g_2((x_1,\ldots,x_{10}))=x_1^2+2x_2^2+\ldots+x_{10}^2-1 \leqslant 0,$$
$$\ldots$$
$$g_{10}((x_1,\ldots,x_{10}))=x_1^2+x_2^2+\ldots+2x_{10}^2-1 \leqslant 0.$$

Мы приведём пример для $n = 10$, начального приближения $x^0 = (1, 1, ..., 1)$ с параметром $\Theta=3$ при выборе стандартной евклидовой прокс-структуры. Координаты точек  $A_k= (a_{1k}, a_{2k}, \ldots, a_{10k})$ при $k=1,2, \ldots, 10$ мы выбираем как строки следующей матрицы $A$:
$$
\left(
\begin{array}{cccccccccc}
1\, & 2\, & 1\, & 4\, &1\, & 0\, & 4\, & 4\, & 4\, & 3\\
2&  4 & 3 & 1  & 0  & 2  & 4  & 0  &4   & 0\\
3& 2  & 3 &  4 & 3  & 0  & 3  & 4  & 2  & 3\\
0&   0&  2&  0 & 2  & 4  & 4  & 1  &0   & 0\\
3&  3 & 4 & 4  & 3  & 0  & 1  & 0  & 4  &4 \\
2& 2  &4  &0   & 4  &0   & 2  & 2  &1   &1 \\
0& 4  & 3 & 4  & 2  & 3  & 3  & 4  & 0  &2 \\
2&  2 & 1 &  4 & 2  & 1  & 4  & 3  & 0  &3 \\
4&  1 & 2 &  2 & 3  & 3  & 2  & 1  & 3  &1 \\
3& 3  & 2 &  2 & 0  & 0  & 4  & 0  & 3  & 4
\end{array}
\right)
$$

Отметим также, что возможно некоторое ускорение метода в случае нескольких ограничений за счёт возможности выбора подходящего ограничения на непродуктивных итерациях (см. алгоритм 6 ниже \cite{bib_Stonyakin}), что видно из таблицы 1 ниже.

\begin{algorithm}
\caption{Модифицированный адаптивный зеркальный спуск}
\label{alg_mod_growth}
\begin{algorithmic}[1]
\REQUIRE $\varepsilon>0,\Theta_0:\,d(x_*)\leqslant\Theta_0^2$
\STATE $x^0=argmin_{x\in X}\,d(x)$
\STATE $I=:\emptyset$
\STATE $N\leftarrow0$
\REPEAT
    \IF{$g(x^N)\leqslant\varepsilon$}
        \STATE $h_N\leftarrow\frac{\varepsilon}{||\nabla f(x^N)||_{*}}$
        \STATE $x^{N+1}\leftarrow Mirr_{x^N}(h_N\nabla f(x^N))\;\text{// \emph{"продуктивные шаги"}}$
        \STATE $N\rightarrow I$
    \ELSE
        \STATE // \emph{$(g_{m(N)}(x^N)>\varepsilon)\;\text{для некоторого}\; m(N)\in \{1,\ldots,M\}$}
        \STATE $h_N\leftarrow\frac{\varepsilon}{||\nabla g_{m(N)}(x^N)||_{*}^2}$
        \STATE $x^{N+1}\leftarrow Mirr_{x^N}(h_N\nabla g_{m(N)}(x^N))\;\text{// \emph{"непродуктивные шаги"}}$
    \ENDIF
    \STATE $N\leftarrow N+1$
\UNTIL{$\Theta_0^2 \leqslant \frac{\varepsilon^2}{2}\left(|I|+\sum\limits_{k\not\in I}\frac{1}{||\nabla g_{m(k)}(x^k)||_{*}^2}\right)$}
\ENSURE $\bar{x}^N:=argmin_{x^k,\;k\in I}\,f(x^k)$
\end{algorithmic}
\end{algorithm}

\begin{table}[H]
	\centering
	\caption{Сравнение алгоритмов 1, 5 и 6}
	\label{table_ex1_alpha2}
	\begin{tabular}{|c|c|c|c|c|c|c|}
		\hline
		\multirow{2}{*}{$\varepsilon$} &  Итерации  & Время, с  &  Итерации  & Время, с  &  Итерации  &  Время, с \\ \cline{2-7}
		& \multicolumn{2}{c|}{Алгоритм 5} & \multicolumn{2}{c|}{Алгоритм 1} & \multicolumn{2}{c|}{Алгоритм 6} \\ \hline
	$1/2$ &1659  &97   &283  &15  &231  &6 \\ \hline
	$1/4$ &5951  &336  &899  &49  &774  &22   \\ \hline
	$1/8$ &22356 &1491 &3159 &180 &2850 &100  \\ \hline
	\end{tabular}
\end{table}

Приведём также сравнение скорости работы методов при тех же параметрах, но уже с негладкими функциональными ограничениями:
$$g_1((x_1,\ldots,x_{10}))= 2|x_1|+|x_2|+|\ldots+x_{10}|-1 \leqslant 0,$$
$$g_2((x_1,\ldots,x_{10}))=|x_1|+3|x_2|+\ldots+|x_{10}|-1 \leqslant 0,$$
$$\ldots$$
$$g_{10}((x_1,\ldots,x_{10}))=|x_1|+|x_2|+\ldots+11|x_{10}|-1 \leqslant 0.$$

\begin{table}[H]
	\centering
	\caption{Сравнение алгоритмов 1, 5 и 6}
	\label{table_ex1_alpha2}
	\begin{tabular}{|c|c|c|c|c|c|c|}
		\hline
		\multirow{2}{*}{$\varepsilon$} &  Итерации  & Время, с  &  Итерации  & Время, с  &  Итерации  &  Время, с \\ \cline{2-7}
		& \multicolumn{2}{c|}{Алгоритм 5} & \multicolumn{2}{c|}{Алгоритм 1} & \multicolumn{2}{c|}{Алгоритм 6} \\ \hline
	$1/2$ &3709  &279  &671  &29  &437  &21 \\ \hline
	$1/4$ &14212 &833  &2418 &103 &1970 &95   \\ \hline
	$1/8$ &54655 &2980 &8979 &455 &8329 &344  \\ \hline
	\end{tabular}
\end{table}

\subsection{О преимуществах использования метода с рестартами в сильно выпуклом случае.} Для демонстрации преимуществ алгоритма \ref{algorithm2} по сравнению с алгоритмом \ref{algorithm1}, был проведен ряд численных экспериментов. Рассмотрим различные $1$-сильно выпуклые целевые функционалы $f$, которые имеют липшицев градиент.
\begin{itemize}
	\item \textbf{Пример 1.}
	$$
	f(x) = \dfrac{L-\mu}{4}\left\{\dfrac{1}{2} \left[ x_{1}^{2} + \sum_{i=1}^{n-1}(x_{i} - x_{i+1})^2  \right] - x_{1} \right\} + \dfrac{\mu}{2} \| x \| ^{2}, \
	$$
	где $\mu = 1, L = 10\,000$ и $n=10 $.
	\item \textbf{Пример 2.}
	$$
	f(x) = \max \{ f_{1}(x), f_{2}(x), f_{3}(x) \}, \text{где}
	$$
	$$
	f_{1}(x) =\dfrac{1}{2} \left(x_{1}^{2} + x_{2}^{2} + 2x_{3}^{2} + 4x_{4}^{2} + x_{5}^{2} + 5x_{6}^{2} + 3x_{7}^{2} + 2x_{8}^{2} + 4x_{9}^{2} + 8x_{10}^{2}\right) - \sum_{i=1}^{10} ix_{i} +5,
	$$
	$$
	f_{2}(x) = \dfrac{1}{2} \left(2x_{1}^{2} + x_{2}^{2} + 3x_{3}^{2} + 4x_{4}^{2} + 2x_{5}^{2} + 5x_{6}^{2} + x_{7}^{2} + 6x_{8}^{2} + 7x_{9}^{2} + 2x_{10}^{2}\right) - \sum_{i=11}^{20} ix_{i} + 6,
	$$
	$$
	f_{3}(x) = \dfrac{1}{2} \left(x_{1}^{2} + x_{2}^{2} + 2x_{3}^{2} + 3x_{4}^{2} + 5x_{5}^{2} + x_{6}^{2} + 4x_{7}^{2} + 2x_{8}^{2} + 3x_{9}^{2} + 6x_{10}^{2}\right) - \sum_{i=21}^{30} ix_{i} + 7.
	$$
	\item  \textbf{Пример 3, \textbf{задача регрессии}} \cite{bib_acclereting_Nes_methods}.
	$$
	f(x)= \dfrac{1}{2} \|Ax-b\|^{2} + \dfrac{\mu}{2}\|x\|^{2}, \text{ где }
	$$
	$$ A=
	\left(
	\begin{array}{cccccccccc}
	5 &  3 &  3 &  5 &  4 &  4 & 3  &3   &5   &1 \\
	2\, & 4\, & 3\, & 5\, & 3\, & 4\, & 2\, & 2\, & 5\, & 4\\
	5& 2 &  1&  4&  1& 1 & 2 & 3 & 5 & 5,
	\end{array}
	\right)
	$$
	при $b= (1,2,3)^T $, $\mu=1$.
	
	\item \textbf{Пример 4.} Рассмотрим функцию следующего вида \cite{bib_acclereting_Nes_methods}:
	$$
	f(x) = \sum_{i=1}^{10} i x_i^{4} + \dfrac{1}{2}\|x\|^2
	$$
	\item \textbf{Пример 5.} Следующий тест выполнен для сглаженной сильно выпуклой версии задачи подавления шумов \cite{bib_acclereting_Nes_methods}
	$$
	f(x)= \dfrac{1}{2} \|Ax-b\|^{2} +\lambda \|x\|_{l_{1},\tau} + \dfrac{\mu}{2}\|x\|^{2}, \text{ где}
	$$
		$$ A=
	\left(
	\begin{array}{cccccccccc}
	9&  2 &  4 &  2 &  2 &  3 & 6  & 3  & 5  & 5\\
	6\, & 7\, & 2\, & 4\, & 8\, & 6\, & 8\, & 8\, & 5\, &1
	\end{array}
	\right),
	b= (1,2)^T, \mu = 1, \lambda = 0.05 , \tau = 0.0001
	$$
	и $\| . \|_{l_{1},\tau}$ задается следующим образом:
	$$
	\| x \|_{l_{1},\tau}=
	\begin{cases}
	|x|-\dfrac{\tau}{2} & \text{if} \ \   |x| \geq \tau\\
	\dfrac{1}{2\tau} x^2 & \text{if} \ \   |x| < \tau
	\end{cases}
	$$
	если $x$ --- скаляр и $\| x \|_{l_{1},\tau}= \sum_{i=1}^{n} \| x_i \|_{l_{1},\tau} $ если $x = (x_1, x_2, ..., x_n)$ --- вектор в $\mathbb{R}^{n}$. Отметим, что квадратичное слагаемое $\dfrac{\mu}{2}\|x\|^{2}$ гарантирует сильную выпуклость целевой функции.
\end{itemize}

Рассмотрим функциональные ограничения вида $g(x) = G(x) + S(x)$, где $S(x) = \dfrac{1}{2}\|x\|^2 $ и
$G(x) = \max\limits_{i \in \overline{1, m}} g_i(x)$, так, что $g_i(x) =\langle \alpha_i , x \rangle + \beta_i $, где $\alpha_i^{T}$ --- строки матрицы
$$
\left(
\begin{array}{cccccccccc}
1 & 1  & 1  & 1  & 1  & 1  & 1  & 1  & 1  & 1\\
7\, & 8\, & 6\, & 2\, & 9\, & 2\, & 3\, & 3\, & 2\, & 6\\
6 & 3 & 4 & 3 & 5 & 1 & 6 & 3 & 2 & 8\\
3 & 5 & 2 & 7 & 8 & 3 & 2 & 1 & 5 & 2\\
2 & 3 & 1 & 8 &  1&  2&  1&  1& 5 & 8\\
1 & 8 & 9 & 1 & 3 &5  & 1 & 3 & 5 & 2\\
1 & 7 & 8 & 5 & 5& 9 & 3 & 1 & 6 & 4\\
7& 3 & 5 & 8 & 9 & 1 & 8 & 7 &8  &8 \\
6 & 4 &  6& 2 & 9 & 2 & 3 & 1 & 6 & 3\\
2 &3  &  4& 4 & 2 & 1 & 9 & 1 &1  & 8
\end{array}
\right)
$$
и константы $\beta_i$ есть нули.

Считаем, что имеется стандартное евклидово расстояние и соответствующая прокс-структура, и
$$X = B_{1}(0) = \left\{x=(x_1,x_2,\dots ,x_{10}) \in \mathbb{R}^{10} \;|\; x_1^2+x_2^2 + ... + x_{10}^2 \leq 1 \right\},$$
начальное приближение $x^{0} = \frac{(1,1,...,1)}{\|(1,1,...,1)\|}$, $\Theta_0 = 3$, $R_0 = 2$, и точность $\varepsilon = 0.05 $.

Результаты выполнения алгоритмов \ref{algorithm1} и \ref{algorithm2} представлены в таблице 3. Приводится количество итераций и время (указано в минутах и в секундах) работы каждого алгоритма \ref{algorithm1} и \ref{algorithm2}.

Все вычисления были произведены с помощью программного обеспечения Python 3.4, на компьютере оснащенном Intel(R) Core(TM) i7-8550U CPU @ 1.80GHz, 1992 Mhz, 4 Core(s), 8 Logical Processor(s). ОЗУ компьютера составляла 8 ГБ.
\begin{table}[H]
	\centering
	\caption{Сравнение результатов работы алгоритмов \ref{algorithm1} и \ref{algorithm2}.}
	\label{table1}
	\begin{tabular}{|c|c|c|c|c|}
		\hline
		\multirow{2}{*}{} & Итерации          & Время       & Итерации         & Время       \\ \cline{2-5}
		& \multicolumn{2}{c|}{Алгоритм \ref{algorithm1}} & \multicolumn{2}{c|}{Алгоритм \ref{algorithm2}}   \\ \hline
		Пример 1  & 115\,973&09:16   &95\,447 &07:37   \\ \hline
		Пример 2  & 57\,798 &07:01   &45\,455 &05:14  \\ \hline
		Пример 3  & 56\,874 &05:02   &50\,747 &04:18   \\ \hline
		Пример 4  & 13\,720 &01:15   &6\,764  &00:38  \\ \hline
		Пример 5  & 64\,324 &06:04   &55\,073 &04:52    \\ \hline
\end{tabular}
\end{table}
Из таблицы \ref{table1}  видно, что алгоритм \ref{algorithm2} работает быстрее алгоритма \ref{algorithm1}.

\textbf{Благодарности.} Авторы выражают огромную признательность Юрию Евгеньевичу Нестерову, Александру Владимировичу Гасникову и Павлу Евгеньевичу Двуреченскому за плодотворные обсуждения и комментарии.

\end{document}